\documentclass[10pt]{article}
\usepackage[all,pdf]{xy}
\usepackage{amsmath}
\usepackage{amsfonts}
\usepackage{amssymb}
\usepackage{amsthm}
\usepackage{latexsym}
\usepackage{indentfirst}
\usepackage{bm}
\usepackage{graphicx}
\usepackage{epsfig}
\usepackage{mathrsfs}

\usepackage{xcolor}
\usepackage{bbm}

\usepackage{enumerate}

\usepackage{algorithm}
\usepackage{algorithmicx}
\usepackage{algpseudocode}
\usepackage{ulem}

\newtheorem{theorem}{Theorem}[section]
\newtheorem{example}[theorem]{Example}
\newtheorem{definition}[theorem]{Definition}
\newtheorem{proposition}[theorem]{Proposition}
\newtheorem{lemma}[theorem]{Lemma}
\newtheorem{corollary}[theorem]{Corollary}

\usepackage{fancyhdr}

\usepackage{tikz-cd}

\usepackage{tikz}
\newcommand*{\circled}[1]{\lower.7ex\hbox{\tikz\draw (0pt, 0pt)%
		circle (.5em) node {\makebox[1em][c]{\small #1}};}}
\usepackage{geometry}
\newcommand{\id}{\mathrm{Id}}

\newcommand{\tip}{\mathrm{Tip}}
\newcommand{\ntip}{\mathrm{NonTip}}
\newcommand{\qa}{\mathbbm{k}Q/I}
\newcommand{\tor}{\mathrm{Tor}}
\newcommand{\ext}{\mathrm{Ext}}
\newcommand{\supp}{\mathrm{Supp}}

\begin{document}
	
	\title{\bf \Large A note on higher structures on the complexes associated to quiver algebras with applications to toupie algebras}
	\author{Yuming Liu$^a$ and Bohan Xing$^{a,*}$}
	\maketitle
	
	\renewcommand{\thefootnote}{\alph{footnote}}
	\setcounter{footnote}{-1} 
	\renewcommand{\thefootnote}{\alph{footnote}}
	\setcounter{footnote}{-1} 
	\setcounter{footnote}{-1} \footnote{$^a$School of Mathematical Sciences, Laboratory of Mathematics and Complex Systems, Beijing Normal University,
		Beijing 100875,  P. R. China. E-mail: ymliu@bnu.edu.cn (Y.M. Liu); bhxing@mail.bnu.edu.cn (B.H. Xing).}
	\setcounter{footnote}{-1} \footnote{$^*$Corresponding author.}
	
	{\noindent\small{\bf Abstract:} In this paper, we summarize a general method of transforming DG structures into higher structures on the various complexes related to the reduced bar resolution of a given quiver algebra using algebraic Morse theory. As an application, we describe the $A_\infty$-structures of toupie algebras. Additionally, for certain special toupie algebras, we also prove that their double homological duals are isomorphic to their associated graded algebras.}

	\section{Introduction}	
	
	It is well-known that higher structures nowadays play important roles in the research of representation theory. For example, Keller provided motivation to use the $A_\infty$-structure to reconstruct a complex by its homology in \cite{Keller}; Tamaroff used the $A_\infty$-coalgebra structure to develop a calculation method for the $A_\infty$-algebra structure of monomial algebras in \cite{TAMA}. Moreover, there is also a focus of study on $L_\infty$-structures since $L_\infty$-algebras are a natural generalization that appears to avoid the rigidity of classical DG Lie algebras and helpful for dealing with deformation problems (see for example in \cite{BW}).
	
	The higher structures that we can construct now mainly come from transfer theories. In  classic transfer theories, a (strong) deformation retract datum is needed to obtain  higher structures on the given complex. A synthesis of these methods can be found in \cite{LV}. More specific formulas can be found in \cite{MM} for the $A_\infty$ case and in \cite{RR} for the $L_\infty$ case.
	 
	One aim of the  present paper is to summarize a general method for transforming DG structures into higher structures on the complexes associated to quiver algebras using algebraic Morse theory, as algebraic Morse theory can provide us a natural strong deformation retract datum from some complex associated to a quiver algebra.
	
	About fifteen years ago, algebraic Morse theory was developed by Kozlov \cite{K}, by Sk\"{o}ldberg \cite{ES}, and by J\"{o}llenbeck and Welker \cite{JW}, independently. Since then, this theory has been widely used in algebra. For additional  references in this area, see the introduction of the recent paper \cite{CLZ}. This theory provides a natural strong deformation retract datum from some complex associated to a quiver algebra to its Morse complex.	
	In particular, Sk{\"o}ldberg \cite{ES} applied this theory to construct the so-called two-sided Anick resolution from the reduced bar resolution of an non-commutative polynomial algebra. Chen, Liu and Zhou \cite{CLZ} generalized the two-sided Anick resolution from non-commutative polynomial algebras to algebras given by quivers with relations. It follows that there exists a strong deformation retract datum from the bar complexes of a quiver algebra to its two-sided Anick resolution. Since we have plenty of DG structures on the bar complex of a quiver algebra, we can obtain the corresponding higher structures through transfer theories (see Proposition \ref{quiver-trans}).
	
	For toupie algebras, which are a generalization of canonical algebras, we apply our method to study their higher structures.  By using algebraic Morse theory, we give the specific strong deformation retract datum between their reduced bar resolutions and their two-sided Anick resolutions (see Lemma \ref{sdr of tor}). Then, we establish the specific $A_\infty$-coalgebra structure on their Tor complexes (see Theorem \ref{co-A}) and the $A_\infty$-algebra structure on their Ext complexes (see Corollary \ref{Toupie-Ainf}) through the transfer theories. This will also give us a specific multiplication structure on the Yoneda algebras of toupie algebras. Since we have verified that toupie algebras are quasi-hereditary and \cite[~Theorem~2.2.1]{CPS} states that if an algebra $A$ is quasi-hereditary under some special conditions, the double homological dual of $A$ is isomorphic to its associated graded algebra $gr(A)$, we also present the specific structure of the associated graded algebras of toupie algebras (see Lemma \ref{grA}) and prove that the double homological duals of toupie algebras are isomorphic to their associated graded algebras in some special cases (see Theorem \ref{double-dual}).

	\medskip
	\textbf{Outline.} \;In Section 2, we recall some commonly used transfer theories for higher structures. In Section 3, we give an SDR datum using algebraic Morse theory and summarize a general method for transforming DG structures into higher structures on the complexes associated to quiver algebras. In section 4, for toupie algebras, we give the specific $A_\infty$-coalgebra structure on their Tor complexes and the $A_\infty$-algebra structure on their Ext complexes (which will also induce the Yoneda algebra structure particularly). Moreover, we present the specific structure of the associated graded algebras of toupie algebras and prove that the double homological duals of toupie algebras are isomorphic to their associated graded algebras in some special cases.
		
	\section{Higher structures and transfer theories}
	
	Now we give a brief introduction about the most commonly used higher structures. For convenience, we fix a field $\mathbbm{k}$ of characteristic zero as ground field and $\otimes=\otimes_\mathbbm{k}$.
		
	\subsection{$A_{\infty}$-structures and transfer theories}
		
	In this section, we review the definition of the $A_\infty$-algebra and the transfer theory.
		
	\begin{definition}$(\cite[~Definition~1.1]{LPWZ})$
		An $A_\infty$-algebra over a base field $\mathbbm{k}$ is a $\mathbb{Z}$-graded vector space $A=\oplus_{p\in\mathbb{Z}}A^p$ endowed with a family of graded $\mathbbm{k}$-linear maps $$m_n:\; A^{\otimes n}\rightarrow A,\quad n\geq 1,$$ 
		of degree $2-n$ satisfying the following Stasheff identities:
		\begin{equation}
			\sum_{r+s+t=n,\;r,t\geq 0,\; s\geq 1}(-1)^{r+st}m_{r+1+t}(\id^{\otimes r}\otimes m_s\otimes \id^{\otimes t})=0 \tag{SI(n)}
		\end{equation}
		for all $n \geq 1$ and $\id$ denotes the identity map of $A$.	
	\end{definition}
			
	Note that when the formulas above are applied to element, additional signs appear due to the Koszul sign rule. Actually, SI(1) is $m_1m_1=0$ and SI(2) is $m_1m_2=m_2(m_1\otimes\id+\id\otimes m_1)$, so $m_1$ is a differential on $A$ and is a graded derivation with respect to $m_2$. Moreover, the identity SI(3) is 
	$$m_2(\id\otimes m_2-m_2\otimes\id)=m_1m_3+m_3(m_1\otimes\id\otimes\id+\id\otimes m_1\otimes\id+\id\otimes\id\otimes m_1).$$
	If either $m_1$ or $m_3$ is zero, then $m_2$ is associative. In general, $m_2$ is associative up to the chain homotopy given by $m_3$. In particular, $A$ is a differential graded (aka. DG) associative algebra if $m_i=0$, $i\geq 3$.
	
	Dually, we can also define the $A_\infty$-structure in a coalgebra version.
	
	\begin{definition}$(\cite[~Section~2.5]{TAMA})$
		An $A_\infty$-coalgebra over a base field $\mathbbm{k}$ is a $\mathbb{Z}$-graded vector space $V=\oplus_{p\in\mathbb{Z}}V^p$ endowed with a family of graded $\mathbbm{k}$-linear maps $$\Delta_n:\; V\rightarrow V^{\otimes n},\quad n\geq 1,$$ 
		of degree $n-2$ satisfying the following Stasheff identities:
		\begin{equation}
			\sum_{r+s+t=n,\;r,t\geq 0,\; s\geq 1}(-1)^{r+st}(\id^{\otimes r}\otimes \Delta_s\otimes \id^{\otimes t})\Delta_{r+1+t}=0 \tag{SI(n)'}
		\end{equation}
		for all $n \geq 1$ and $\id$ denotes the identity map of $V$. 	
	\end{definition}
			
	Moreover, the sequence of maps be locally finite means that for each element $v\in V$, the set $\{\Delta_n(v)\;|\; n\geq 1\}$ contains finitely many nonzero terms. We also call $V$  a DG associative coalgebra if $\Delta_i=0$ for $i\geq 3$.
	
	Now we recall the transfer theory for $A_\infty$-structure in \cite{MM}, which can also be found in \cite{LV}. 
	
	\begin{definition}
		Given chain complexes $W_1,W_2$ and chain maps $p:\;W_1\rightarrow W_2$, $i:\; W_2\rightarrow W_1$. We say that $W_2$ is a deformation retract of $W_1$, if there is a homotopy retract datum 
		
		$$
		\begin{tikzcd}
			W_1 \arrow[r, "p", shift left=2] \arrow["h"', loop, distance=2em, in=215, out=145] & W_2 \arrow[l, "i", shift left]
		\end{tikzcd}$$
		with $\id_{W_1}-ip=dh+hd$ and $\id_{W_2}=pi$. Moreover, we call this datum a strong deformation retract datum (aka. SDR datum), if $h^2=0, hi=0, ph=0$.
	\end{definition}
			
	There are also two versions of the transfer theory for $A_\infty$-structure.
	
	\begin{theorem}$(\cite[~Theorem~5]{MM})$
		Let $(A,m_1,m_2)$ be a DG associative algebra and consider a homotopy retract datum 
			$$
		\begin{tikzcd}
			A \arrow[r, "p", shift left=2] \arrow["h"', loop, distance=2em, in=215, out=145] & A' \arrow[l, "i", shift left]
		\end{tikzcd}$$
		There exists an $A_\infty$-algebra structure on $A'$ which is given by $m_1'=d_{A'}$ and, for $n\geq 2$, by $m_n'=pm_ni^{\otimes n}$. To be specific, for $n\geq 3$, $m_n:\;A^{\otimes n}\rightarrow A$ are defined by 
		$$m_n=\sum_{s+t=n,\; s,t>0}(-1)^{s(t+1)}m_2(hm_s\otimes hm_t),$$
		with the convention that $hm_1=\id$.
	\end{theorem}
			
	\begin{theorem}$(\cite[~Theorem~2.4]{TAMA})$\label{trans coalg}
		Let $(V,\Delta_1,\Delta_2)$ be a DG associative coalgebra and consider a homotopy retract datum 
		$$
		\begin{tikzcd}
			V \arrow[r, "p", shift left=2] \arrow["h"', loop, distance=2em, in=215, out=145] & V' \arrow[l, "i", shift left]
		\end{tikzcd}$$
		There exists an $A_\infty$-coalgebra structure on $V'$ which is given by $\Delta_1'=d_{V'}$ and, for $n\geq 2$, by $\Delta_n'=p^{\otimes n}\Delta_ni$. To be specific, for $n\geq 3$, $\Delta_n:\;V\rightarrow V^{\otimes n}$ are defined by 
		$$\Delta_n=\sum_{s+t=n,\; s,t>0}(-1)^{s(t+1)}(\Delta_sh\otimes \Delta_th)\Delta_2,$$
		with the convention that $\Delta_1h=\id$.
	\end{theorem}

	\subsection{$L_{\infty}$-structure and transfer theory}		
	
	In this section, we introduce the notions of $L_\infty$-algebra. Let $V$ be a $\mathbb{Z}$-graded vector space. Denote by $\wedge V$ the free graded-commutative associative algebra over $V$. Let $v_1,\cdots,v_n\in V$ be homogeneous elements and $v_1\wedge \cdots\wedge v_n$ be its product in $\wedge V$.
	
	Let $\sigma\in \mathbb{S}_n$ be a permutation and $v_1,\cdots,v_n\in V$ homogeneous elements. The Koszul sign $\epsilon(\sigma)$=$\epsilon(\sigma;v_1,\cdots,v_n)\in\{\pm1\}$ is defined by $$v_1\wedge \cdots\wedge v_n=\epsilon(\sigma)v_{\sigma(1)}\wedge \cdots\wedge v_{\sigma(n)}.$$ The antisymmetric Koszul sign $\chi(\sigma)=\chi(\sigma;v_1,\cdots,v_n)\in\{\pm1\}$ is defined by $$\chi(\sigma)=sgn(\sigma)\epsilon(\sigma;v_1,\cdots,v_n)$$ where $sgn(\sigma)$ denotes the sign of the permutation $\sigma$. For any $t$ with $1\leq t\leq n$, $\kappa(\sigma)_t=\kappa(\sigma,t;v_1,\cdots,v_n)\in\{\pm1\}$ is given by 
	$$\kappa(\sigma)_t=(-1)^{(t-1)+(n-t-1)(\sum_{i=1}^{t}|v_{\sigma(p)}|)}$$
	where we write $|v_{\sigma(p)}|$ for the degree of the homogeneous element $v_{\sigma(p)}$.
	
	Let $\mathbb{S}_{t,n-t}$ be the set of all $(t,n-t)$-unshuffles in the symmetric group $\mathbb{S}_n$, that is,
	$$\mathbb{S}_{t,n-t}=\{\sigma\in\mathbb{S}_n\;|\; \sigma(1)<\cdots<\sigma(t),\; \sigma(t+1)<\cdots<\sigma(n)\}.$$
	Let $\mathbb{S}^-_{t,n-t}$ be the set of all $\sigma\in\mathbb{S}_{t,n-t}$ such that $\sigma(1)<\sigma(t+1)$. Analogously, for $i+j+k=n$ we define
	$$\mathbb{S}_{i,j,k}=\{\sigma\in\mathbb{S}_n\;|\; \sigma(1)<\cdots<\sigma(i),\;\sigma(i+1)<\cdots<\sigma(i+j),\; \sigma(i+j+1)<\cdots<\sigma(n)\}$$
	and similarly we define $\mathbb{S}^-_{i,j,k}$ be the set of all $\sigma\in\mathbb{S}_{i,j,k}$ such that $\sigma(1)<\sigma(i+1)<\sigma(i+j+1)$.
	
	\begin{definition}$(\cite[~Definition~2.2]{RR})$
		An $L_\infty$-algebra is a $\mathbb{Z}$-graded vector space $L$ together with linear maps $$l_n:\; L^{\otimes n}\rightarrow L$$ of degree $2-n$ such that, for every $n\in\mathbb{N}$ and homogeneous elements $v_1,\cdots,v_n\in L$, the following conditions are satisfied:
		$$l_n\hat{\sigma}=\chi(\sigma)l_n,\quad \forall\sigma\in\mathbb{S}_n;$$
		\begin{equation}
			\sum_{i+j=n+1}\sum_{\sigma\in\mathbb{S}_{i,n-i}}(-1)^{i(j-1)}\chi(\sigma)l_j(l_i\otimes\id^{\otimes n-i})\hat{\sigma}=0, \tag{JI(n)}
		\end{equation}
		for all $n \geq 1$ where $\hat{\sigma}(v_1\otimes\cdots\otimes v_n)=v_{\sigma(1)}\otimes\cdots\otimes v_{\sigma(n)}$ and $\id$ denotes the identity map of $L$.
	\end{definition}
	
	Observe that for $n=2$, the equations in the previous definition are
	$$l_2(v_2\otimes v_1)=-(-1)^{|v_1||v_2|}l_2(v_1\otimes v_2),$$
	$$l_1l_2(v_1\otimes v_2)=l_2(l_1(v_1)\otimes v_2)+(-1)^{|v_1|}l_2(v_1\otimes l_1(v_2)).$$
	Moreover, when $n=3$ and $l_3=0$, we get the Jacobi identity
	$$\sum_{\sigma\in\mathbb{S}_{2,1}}\chi(\sigma)l_2(l_2\otimes \id)\hat{\sigma}=0.$$
	We also call $L$  a DG Lie algebra if $l_i=0$ for $i\geq 3$.
	
	\begin{theorem}$(\cite[~Theorem~2.7]{RR})$
		Let $(L,d,[-,-])$ be a DG Lie algebra and consider a SDR datum 
		$$
		\begin{tikzcd}
			L \arrow[r, "p", shift left=2] \arrow["h"', loop, distance=2em, in=215, out=145] & L' \arrow[l, "i", shift left]
		\end{tikzcd}$$	
		Define the linear maps of degree $2-n$
		$$l_n:\;L'^{\otimes n}\rightarrow L',\quad v_n:\; L'^{\otimes n}\rightarrow L$$
		and the linear maps of degree $1-n$
		$$\phi_n:\;L'^{\otimes n}\rightarrow L$$
		by $l_1=d_{L'}$, $v_1=0$, $\phi_1=i$. For $n>1$ and homogeneous elements $f_1,\cdots,f_n\in L'$, we define the map as follows:
		$$v_n=\sum_{t=1}^{n-1}\sum_{\tau\in\mathbb{S}_{t,n-t}^-}\chi(\tau)\kappa(\tau)_t[\phi_t,\phi_{n-t}]\hat{\tau},$$
		$$\phi_n=hv_n,$$
		$$l_n=pv_n,$$
		where $$\hat{\tau}(f_1\otimes\cdots\otimes f_n)=f_{\tau(1)}\otimes\cdots\otimes f_{\tau(n)},\;\; and$$
		$$[\phi_t,\phi_{n-t}](f_1\otimes\cdots\otimes f_n)=[\phi_t(f_1\otimes\cdots\otimes f_t),\phi_{n-t}(f_{t+1}\otimes\cdots\otimes f_n)].$$
		Then $(L',\{l_n\}_{n\geq 1})$ gives an $L_\infty$-algebra structure on $L'$.
	\end{theorem}
	
		\section{Higher structures on the complexes associated to quiver algebras}	
		
		In this sections, we use algebraic Morse theory to construct an SDR datum of the complexes associated to quiver algebras. Then we can use the dg structure on them and the transfer theory we discussed above to get the higher structures on the Morse complexes.
		
		\subsection{An SDR datum from algebraic Morse theory}

		The most general version of algebraic Morse theory was presented in Chen, Liu and Zhou \cite{CLZ}. For our purpose, we will adopt to a Morse matching condition defined in \cite[Proposition 3.2]{CLZ}.
		
		Let $R$ be an associative ring and $C_{*}=(C_n, \partial_n)_{n\in\mathbb{Z}}$ be a chain complex of left $R$-modules. We assume that each $R$-module $C_n$ has a decomposition $C_n\simeq\oplus_{i\in I_n}C_{n,i}$ of $R$-modules, so we can regard the differentials $\partial_n$ as a matrix $\partial_n=(\partial_{n,ji})$ with $i\in I_n$ and $j\in I_{n-1}$ and where $\partial_{n,ji}:C_{n,i}\rightarrow C_{n-1,j}$ is a homomorphism of $R$-modules.
		
		Given the complex $C_*$ as above, we construct a weighted quiver $G(C_*):=(V,E)$. The set $V$ of vertices of $G(C_*)$ consists of the pairs $(n,i)$ with $n\in\mathbb{Z},i\in I_n$ and the set $E$ of weighted arrows is given by the rule:
		if the map $\partial_{n,ji}$ does not vanish, draw an arrow in E from $(n,i)$ to $(n-1,j)$ 
		and denote the weight of this arrow by the map $\partial_{n,ji}$.
		
		A full subquiver $\mathcal{M}$ of the weighted quiver $G(C_*)$ is called a partial matching if it satisfies the following two conditions:
		\begin{itemize}
			\item $(Matching)$ Each vertex in $V$ belongs to at most one arrow of $\mathcal{M}$.
			
			\item $(Invertibility)$ Each arrow in $\mathcal{M}$ has its weight invertible as a $R$-homomorphism.
			
		\end{itemize}
		With respect to a partial matching $\mathcal{M}$, we can define a new weighted quiver $G_{\mathcal{M}}(C_*)=(V,E_{\mathcal{M}})$, where $E_{\mathcal{M}}$ is given by 
		\begin{itemize}
			\item Keep everything for all arrows which are not in $\mathcal{M}$ and call them thick arrows.
			
			\item For an arrow in $\mathcal{M}$, replace it by a new dotted arrow in the reverse direction and the weight of this new arrow is the negative inverse of the weight of original arrow.
		\end{itemize}
		A path in  $G_{\mathcal{M}}(C_*)$ is called a zigzag path if dotted arrows and thick arrows appear alternately.
		
		Next, for convenience, we will introduce from J\"{o}llenbeck and Welker \cite{JW} the notations related to the weighted quiver $G(C_*)=(V,E)$ with a partial matching $\mathcal{M}$ on it.
		
		\begin{definition} 
			\begin{enumerate}[(1)]
				\item A vertex $(n,i)\in V$ is critical with respect to $\mathcal{M}$ if $(n,i)$ does not lie in any arrow in $\mathcal{M}$. Let $V_n$ denote all the vertices with the first number equal to $n$, we write
				$$V_n^{\mathcal{M}}:=\{(n,i)\in V_n\;|\;(n,i)\;is \; critical\}$$
				for the set of all critical vertices of homological degree $n$.
				
				\item Write $(m,j)\leq (n,i)$ if there exists an arrow from $(n,i)$ to $(m,j)$ in $G(C_*)$.
				
				\item Denote by $P((n,i),(m,j))$ the set of all zigzag paths from $(n,i)$ to $(m,j)$ in  $G_{\mathcal{M}}(C_*)$.
				
				\item The weight $w(p)$ of a path $$p=((n_1,i_1)\rightarrow (n_2,i_2)\rightarrow \cdots\rightarrow (n_r,i_r))\in P((n_1,i_1),(n_r,i_r))$$ in  $G_{\mathcal{M}}(C_*)$ is given by
				
				$$w(p):=w((n_{r-1},i_{r-1})\rightarrow (n_{r},i_{r}))\circ\cdots\circ w((n_1,i_1)\rightarrow (n_2,i_2)),$$
				
				$$
				w((n,i)\rightarrow (m,j)):=\left\{
				\begin{array}{*{3}{lll}}
					-\partial_{m,ij}^{-1}&,& (n,i)\leq (m,j),\\
					\partial_{n,ji}&,& (m,j)\leq (n,i).
				\end{array}
				\right.
				$$
				Then we write $\Gamma((n,i), (m,j))=\sum_{p\in P((n,i),(m,j))}w(p)$ for the sum of weights of all zigzag paths from $(n,i)$ to $(m,j)$.
				
			\end{enumerate}
		\end{definition}
		
		Following \cite[Proposition 3.2]{CLZ}, we call a partial matching $\mathcal{M}$ as above a Morse matching if any zigzag path starting from $(n,i)$ is of finite length for each vertex $(n,i)$ in $G_{\mathcal{M}}(C_*)$. 
		
		Now we can define a new complex $C_*^{\mathcal{M}}$, which we call the Morse complex of $C_*$ with respect to $\mathcal{M}$. The complex $C_*^{\mathcal{M}}=(C_n^{\mathcal{M}},\partial_n^{\mathcal{M}})_{n\in\mathbb{Z}}$ is defined by
		
		$$C_n^{\mathcal{M}}:=\oplus_{(n,i)\in V_n^{\mathcal{M}}}C_{n,i},$$
		
		$$
		\partial_n^{\mathcal{M}}:\left\{
		\begin{array}{*{3}{lll}}
			C_n^{\mathcal{M}}& \rightarrow& C_{n-1}^{\mathcal{M}}\\
			x\in C_{n,i}& \mapsto& \sum_{(n-1,j)\in V_{n-1}^{\mathcal{M}}}\Gamma((n,i),(n-1,j))(x).
		\end{array}
		\right.
		$$
		
		The main theorem of algebraic Morse theory can be stated as follows.
		
		\begin{theorem} $(\cite[~Theorem~3.3]{CLZ})$\label{morse}
			$C_*^{\mathcal{M}}$ is a complex of left $R$-modules which is homotopy equivalent to the original complex $C_*$. Moreover, the maps defined below 			
			$$
			p:\left\{
			\begin{array}{*{3}{lll}}
				C_n& \rightarrow& C_n^{\mathcal{M}}\\
				x\in C_{n,i}& \mapsto& \sum_{(n,j)\in V_{n}^{\mathcal{M}}}\Gamma((n,i),(n,j))(x),
			\end{array}
			\right.
			$$			
			$$
			i:\left\{
			\begin{array}{*{3}{lll}}
				C_n^{\mathcal{M}}& \rightarrow& C_n\\
				x\in C_{n,i}& \mapsto& \sum_{(n,j)\in V_{n}}\Gamma((n,i),(n,j))(x),
			\end{array}
			\right.
			$$
			$$
			h:\left\{
			\begin{array}{*{3}{lll}}
				C_n& \rightarrow& C_{n+1}\\
				x\in C_{n,i}& \mapsto& \sum_{(n+1,j)\in V_{n+1}}\Gamma((n,i),(n+1,j))(x)
			\end{array}
			\right.
			$$	
			give an SDR datum between the complexes $C_*$ and $C_*^\mathcal{M}$:
			$$
			\begin{tikzcd}
				C \arrow[r, "p", shift left=2] \arrow["h"', loop, distance=2em, in=215, out=145] & C^\mathcal{M} \arrow[l, "i", shift left]
			\end{tikzcd}$$				
		\end{theorem}
		
		By giving a suitable Morse matching, Bardzell's projective resolution in \cite{Bardzell}, two-sided Anick resolution in \cite{CLZ} and Chouhy-Solotar projective resolution in \cite{CS} can all be regarded as a Morse complex by using a Morse matching (see for example, in \cite[~Section~4]{CLZ}) from the reduced bar resolution of the quiver algebras. By using the transfer theory above, we have the following theorem.
		
		\begin{theorem}\label{Morse-trans}
			Let $R$ be an associative ring, $C_{*}$ be a chain complex of left $R$-modules, $\mathcal{M}$ be a Morse matching of the weighted quiver $G(C_*)$ and $C_*^\mathcal{M}$ be the Morse complex. If $C_*$ has some operations that make it a DG associative algebra (DG associative coalgebra, DG Lie algebra), it will induce $C_*^\mathcal{M}$ to become an $A_\infty$-algebra ($A_\infty$-coalgebra, $L_\infty$-algebra).
		\end{theorem}
		
		\subsection{DG structures on the complexes associated to quiver algebras}

		Now we will concentrate on quiver algebras of the form $\mathbbm{k}Q/I$, where $\mathbbm{k}$ is a field, $Q$ is a finite quiver, $I$ is a two-sided ideal in the path algebra $\mathbbm{k}Q$. For each integer $n\geq 0$, we denote by $Q_n$ the set of all paths of length $n$ and by $Q_{\geq n}$ the set of all paths with length at
		least $n$. We shall assume that the ideal $I$ is contained in $\mathbbm{k}Q_{\geq 2}$ so that $\mathbbm{k}Q/I$ is not necessarily finite dimensional. We denote by $s(p)$ the source vertex of a path $p$ and by $t(p)$ its terminus vertex. We will write paths from left to right, for example, $p=\alpha_{1}\alpha_{2}\cdots\alpha_{n}$ is a path with starting arrow $\alpha_{1}$ and ending arrow $\alpha_{n}$.  
		The length of a path $p$ will be denoted by $l(p)$. Denote the set of paths in quiver $Q$ also by $Q$ For any $a\in kQ$, we have $a=\sum_{p\in Q,\;\lambda_p\in k}\lambda_p p$ and write $\mathrm{Supp}(a)=\{p\mid \lambda_p\neq 0\}$.
		
		For a quiver algebra $A=\mathbbm{k}Q/I$, denote $E=\mathbbm{k}Q_0$ and $\overline{A}$ is an $E^e$-direct summand of $A$ such that $A=E\oplus \overline{A}$. The reduced bar resolution of the quiver algebra $A$ can be written by the form of the following theorem in the sense of Cibils \cite{C}.
		
		\begin{theorem}
			For the algebra $A=\mathbbm{k}Q/I$, the reduced bar resolution $\mathbb{B}_*$ is a two-sided projective resolution of $A$ with $\mathbb{B}_0=A\otimes_EA$, $\mathbb{B}_n=A\otimes_E(\overline{A})^{\otimes_En}\otimes_EA$ and the differential $d=(d_n)$ is
			
			$$d_n([a_1|\cdots|a_n])=a_1[a_2|\cdots|a_n]+\sum_{i=1}^{n-1}(-1)^i[a_1|\cdots|a_ia_{i+1}|\cdots|a_n]+(-1)^n[a_1|\cdots|a_{n-1}]a_n$$	
			with $[a_1|\cdots|a_n]:=1\otimes a_1\otimes\cdots\otimes a_n\otimes 1$. By convention $\mathbb{B}_{-1}=A$, and $d_0:A\otimes_EA\longrightarrow A$ is given by the multiplication $\mu_A$ in $A$.
		\end{theorem}
		
		Let $\mathbb{P}$ be a projective resolution of $A$ as an $A^e$-module with differential $d_\mathbb{P}$. The total complex $\mathrm{Tot}(\mathbb{P}\otimes_A\mathbb{P})$ is also an $A^e$-projective resolution of $A$ with its $n$-th module $\mathrm{Tot}(\mathbb{P}\otimes_A\mathbb{P})_n$ is given as $\sum_{i+j=n}\mathbb{P}_i\otimes_A\mathbb{P}_j$ and differentials $d_\mathbb{P}\otimes\id+\id\otimes d_\mathbb{P}$ (see for example, \cite[~Page~33]{W}). By comparison theorem, there is a chain map $\Delta_\mathbb{P}:\;\mathbb{P}\rightarrow\mathbb{P}\otimes_A\mathbb{P}$ lifting the canonical isomorphism from $A$ to $A\otimes_A A$. In particular, the diagram
		$$
		\begin{tikzcd}
			{\mathbb{P}} && A \\
			{\mathbb{P}\otimes_A\mathbb{P}} && {A\otimes_AA}
			\arrow[from=1-1, to=1-3]
			\arrow["{\Delta_\mathbb{P}}"', from=1-1, to=2-1]
			\arrow["\cong", from=1-3, to=2-3]
			\arrow[from=2-1, to=2-3]
		\end{tikzcd}$$
		is commutative. The map $\Delta_\mathbb{P}$ is called a diagonal map and it is unique up to chain homotopy. Moreover, we also have $(d_\mathbb{P}\otimes\id+\id\otimes d_\mathbb{P})\Delta_\mathbb{P}=\Delta_\mathbb{P}d_\mathbb{P}$.
		
		Restrict the discussion on the reduced bar resolution $\mathbb{B}_*$, one definition of a diagonal map $\Delta:\;\mathbb{B}_*\rightarrow\mathbb{B}_*\otimes_A\mathbb{B}_*$ is following:
		$$\Delta([a_1|\cdots|a_n])=\sum_{i=0}^{n}[a_1|\cdots|a_i]\otimes[a_{i+1}|\cdots|a_n].$$
		with $[a_0]=[a_{n+1}]=1\otimes 1$. It is obvious that $(\Delta\otimes\id)\Delta=(\id\otimes\Delta)\Delta$. Thus $(\mathbb{B}_*,d_*,\Delta)$ is a DG associative coalgebra. By applying the functor $\mathrm{Hom}_{A^e}(-,A)$, \cite{Ger} has told us there exist two operations $\smile$ and $[-,-]$ on $\mathbb{B}^*=\mathrm{Hom}_{A^e}(\mathbb{B}_*,A)$, such that $(\mathbb{B}^*,d^*,\smile )$ is a DG associative algebra and $(\mathbb{B}^*[1],d^*,[-,-])$ is a DG Lie algebra. Moreover, for $f\in \mathbb{B}^m$ and $g\in\mathbb{B}^n$, we have $f\smile g=\mu_A(f\otimes g)\Delta$ (See for example in \cite{Volkov}).
		
		There are also some DG structure different from above cases we cared about. Consider the bar construction of $A$ which is a complex $\mathbb{B}'_*$ with $\mathbb{B}'_n=E\otimes_A\mathbb{B}_n\otimes_A E\cong\overline{A}^{\otimes n}$ ($n\geq 1$) and the differential
		$$d_n([a_1|\cdots|a_n])=\sum_{i=1}^{n-1}(-1)^{i}[a_1|\cdots|a_ia_{i+1}|\cdots|a_n].$$
		Moreover, the complex $\mathbb{B}'_*$ also admits a diagonal $\Delta'$ given by
		$$\Delta'([a_1|\cdots|a_n])=\sum_{i=1}^{n-1}[a_1|\cdots|a_i]\otimes[a_{i+1}|\cdots|a_n].$$
		Thus $(\mathbb{B}'_*,d_*,\Delta')$ is a DG associative coalgebra. By applying the functor $(-)^\vee=\mathrm{Hom}_\mathbbm{k}(-,\mathbbm{k})$ and \cite[~Section~1.2.2]{LV}, denote by $\mathbb{B}'^*=(\mathbb{B}'_*)^\vee\cong\mathrm{Hom}_A(\mathbb{B}\otimes_A E,E)$ and $D^2$ be the canonical map from $(\mathbb{B}'_*)^\vee\otimes(\mathbb{B}'_*)^\vee$ to $(\mathbb{B}'_*\otimes\mathbb{B}'_*)^\vee$, then $(\mathbb{B}'^*,d,\Delta^\vee\circ D^2)$ is a DG associative algebra.
		
		Now we consider a Morse matching $\mathcal{M}$ of the weighted quiver $G(\mathbb{B}_*)$, it will give us a Morse complex $C^\mathcal{M}_*$ from $\mathbb{B}_*$ by Theorem \ref{morse}. Therefore, by Theorem \ref{Morse-trans}, we will get some higher structure on the complexes associated to quiver algebras.
		
		\begin{proposition}\label{quiver-trans}
			Let $A=\mathbbm{k}Q/I$ be a quiver algebra and $\mathbb{B}_*$ be the complex of reduced bar resolution of $A$. If $\mathcal{M}$ is a Morse matching of the weighted quiver $G(\mathbb{B}_*)$ and $\mathbb{B}^\mathcal{M}_*$ is the Morse complex from $\mathbb{B}_*$, then
			\begin{enumerate}[(1)]
				\item there exists $\{\Delta_i\}_{i\geq 1}$, such that $(\mathbb{B}^\mathcal{M}_*,\{\Delta_i\}_{i\geq 1})$ become an $A_\infty$-coalgebra;
				
				\item there exists $\{m_i\}_{i\geq 1}$, such that $(\mathrm{Hom}_{A^e}(\mathbb{B}^\mathcal{M}_*,A),\{m_i\}_{i\geq 1})$ become an $A_\infty$-algebra;
			
				\item there exists $\{\Delta_i'\}_{i\geq 1}$, such that $(E\otimes_A \mathbb{B}^\mathcal{M}_*\otimes_A E,\{\Delta_i'\}_{i\geq 1})$ become an $A_\infty$-coalgebra;
					
				\item there exists $\{m_i'\}_{i\geq 1}$, such that $(\mathrm{Hom}_A(\mathbb{B}^\mathcal{M}_*\otimes_AE,E),\{m_i'\}_{i\geq 1})$ become an $A_\infty$-algebra;
					
				\item there exists $\{l_i\}_{i\geq 1}$, such that $(\mathrm{Hom}_{A^e}(\mathbb{B}^\mathcal{M}_*,A)[1],\{l_i\}_{i\geq 1})$ become an $L_\infty$-algebra.
					
			\end{enumerate}
		\end{proposition}
		
		For example, \cite{BW} discussed the $L_\infty$-structure in (5) of quiver algebras using the SDR dutum given by Chouhy-Solotar projective resolution of quiver algebras; \cite{RR} discussed the $L_\infty$-structure in (5) of quiver algebras using the SDR dutum given by the Bardzell's minimal projective resolution of monomial algebras; \cite{TAMA} discussed the $A_\infty$-structure in (3) and (4) of quiver algebras using the SDR dutum given by the two-sided Anick resolution of monomial algebras.

	\section{Applications to toupie algebras}

	\subsection{Basic definitions}

	Now we restrict our discussion on a special case of quiver algebras. Recall the definition.
	
	\begin{definition}$(\cite[~Definition~1]{ALS})$
		A finite quiver is toupie if it has a unique source and a unique sink and any other vertex is the source of exactly one arrow and the target of exactly one arrow. The source and sink will be denoted by $0$ and $\omega$, respectively.
		
		Given a toupie quiver $Q$ and any admissible ideal $I\subseteq \mathbbm{k}Q$, $A=\mathbbm{k}Q/I$ will be called a toupie algebra. The paths from $0$ to $\omega$ will be called branches. The $j$-th branches will be denoted by $\alpha^{(j)}$ and the $i$-th arrow in $\alpha^{(j)}$ will be denoted by $\alpha^{(j)}_i$.
	\end{definition}

	By the definition above, for a fixed toupie algebra $A=\mathbbm{k}Q/I$, there are four possible kind of branches in $Q$:
	
	\begin{enumerate}[(1)]
		\item $B_1=\{\alpha^{(1)},\cdots,\alpha^{(a)}\}$: arrows from $0$ to $\omega$;
		
		\item $B_2=\{\alpha^{(a+1)},\cdots,\alpha^{(a+l)}\}$: branches of length greater than $1$ and not involved in any relations;
		
		\item $B_3=\{\alpha^{(a+l+1)},\cdots,\alpha^{(a+l+m)}\}$: branches containing monomial relations;
		
		\item $B_4=\{\alpha^{(a+l+m+1)},\cdots,\alpha^{(a+l+m+n)}\}$: branches involved in non monomial relations. For convenience, let $l(\alpha^{(a+l+m+i)})\geq l(\alpha^{(a+l+m+j)})$ for all $i\leq j$.
		
	\end{enumerate}

	Given a finite set of $r$ equations generating the non monomial relations in $I$ and having fixed the order of branches defined in $B_4$, let $C=(a_{ij})\in\mathbbm{k}^{r\times n}$ be the matrix whose rows are coefficients of each of these equations. Replacing the given relations by those obtained from the reduced row echelon form of the matrix gives of course the same algebra. From now on we will always suppose that this matrix is already reduced and its rank is $r$. Every non monomial relation will be of the form:
	$$\rho_i=\alpha^{(a+l+m+k_i)}+\sum_{j>k_i}a_{ij}\alpha^{(a+l+m+j)},\quad 1\leq i\leq r.$$
	We will call $\tip(\rho_i)=\alpha^{(a+l+m+k_i)}$, $\tip^{-1}(\alpha^{(a+l+m+k_i)})=\rho_i$ and $f_{\rho_i}=-\sum_{j>k_i}a_{ij}\alpha^{(a+l+m+j)}$. Moreover, denote the coefficients of the branches by $c(\alpha^{(a+l+m+j)})=a_{ij}$ $(j>k_i)$ and $c(\alpha^{(a+l+m+k_i)})=1$ in $\rho_i$.
	
	Let $\mathcal{G}$ be a minimal set of generators of $I$ containing all $\rho_i$. The set $\mathcal{G}$ will be the disjoint union of $\mathcal{G}_{mon}$, consisting of monomial relations, and $\mathcal{G}_{nomon}=\{\rho_i\;|\;1\leq i\leq r\}$, the set of non monomial relations. Denote by $\tip\mathcal{G}_{nomon}=\{\tip(\rho_i)\;|\;1\leq i\leq r\}$, $\tip\mathcal{G}=\mathcal{G}_{mon}\cup\mathcal{G}_{nomon}$ and $\ntip(I)=Q\backslash Q\tip\mathcal{G}Q$. Thus, actually, $\mathcal{G}$ is a reduced Gr\"{o}bner basis of $I$ in $\mathbbm{k}Q$ (the definition of Gr\"{o}bner basis can be found in \cite{CLZ}).

	\subsection{The minimal resolution of toupie algebras}
	
			Let $A=\qa$ be a toupie algebra, $\mathcal{G}$ be the set defined above, and denote $W:=\mathrm{Tip}(\mathcal{G})$. Denote by $\mathbb{B}_*$ the reduced bar resolution of $A$. Similar as in \cite{CLZ}, we define a new quiver $Q_W=(V,E)$ with respect to $W$, which is called the Ufnarovski\u\i~graph (or just Uf-graph).
	
	\begin{definition} \label{Uf-graph}
		A (left) Uf-graph $Q_W=(V,E)$ with respect to $W$ of the algebra $A$ is given by
		
		$$V:=Q_0\cup Q_1\cup \{u\in Q_{\geq 0}\ |\ \text{$u$ is a proper right factor of some $v\in W$}\}$$
		
		$$E:=\{e\rightarrow x\ |\ e\in Q_0,\;x=ex\in Q_1\}~~\cup$$$$\{u\rightarrow v\ |\ uv\in\langle \mathrm{Tip}(\mathcal{G})\rangle, \text{ but } w\notin\langle \mathrm{Tip}(\mathcal{G})\rangle\;for\;uv=wp,\;l(p)\geq 1\}$$
	\end{definition}
	
	Using Uf-graph $Q_W$ one can define (for each $i\geq -1$) the $i$-chains, which form a subset of generators of $B_i(A)=\bigoplus A^e[w_1|\cdots|w_i]$ for $i\geq 0$, with $w_1,\cdots,w_i\in \mathrm{NonTip}(I)\backslash Q_0$ and $w_1\cdots w_i\in Q_{\geq 0}$.
	
	\begin{definition} \label{i-chains}
		
		\begin{itemize}
			\item The set $W^{(i)}$ of (left) $i$-chains consists of all sequences $[w_1|\cdots|w_{i+1}]$ with each $w_k\in \mathrm{NonTip}(I)\backslash Q_0$, such that
			$$e\rightarrow w_1\rightarrow w_2\rightarrow \cdots\rightarrow w_{i+1}$$
			is a path in $Q_W$. And define $W^{(-1)}:=Q_0$.
			
			\item For for all $ p\in Q_{\geq 0}$, define $$V_{p,i}^{(n)}=\{[w_1|\cdots|w_n]\ |\ p=w_1\cdots w_n,\; [w_1|\cdots|w_{i+1}]\in W^{(i)},\;[w_1|\cdots|w_{i+2}]\notin W^{(i+1)}\}$$
		\end{itemize}
	\end{definition}

	By using the definition above, we can define a partial matching $\mathcal{M}$ to be the set of arrows of the following form in the weighted quiver $G(\mathbb{B}_*)$, where $\mathbb{B}_*$ is the reduced bar resolution of the toupie algebra $A=kQ/I$:
	
	$$[w_1|\cdots|w_{i+1}|w_{i+2}'|w_{i+2}''|w_{i+3}|\cdots|w_n]\stackrel{(-1)^{i+2}}{\longrightarrow}[w_1|\cdots|w_{i+2}|\cdots|w_n]$$
	where $$w=w_1\cdots w_n=w_1\cdots w_{i+2}'w_{i+2}''\cdots w_n,~~w_{i+2}=w_{i+2}'w_{i+2}'',$$ 
	$$[w_1|\cdots|w_{i+1}|w_{i+2}'|w_{i+2}''|w_{i+3}|\cdots|w_n]\in V_{w,i+1}^{(n+1)},~~[w_1|\cdots|w_{i+2}|\cdots|w_n]\in V_{w,i}^{(n)}.$$
	
	\begin{theorem}$(\cite[~Theorem~4.3]{CLZ})$\label{morse matching}
		The partial matching $\mathcal{M}$ is a Morse matching of $G(\mathbb{B}_*)$ with the set of critical vertices in $n$-th component is $W^{(n-1)}$.
	\end{theorem}
	
	By the definition of $n$-chains, we have an one-to-one correspondence between $[w_0|\cdots|w_n]\in W^{(n)}$ and paths $w_0\cdots w_n\in Q$ which have the same property in Definition \ref{i-chains}. Therefore, we also regard $W^{(n)}$ as a subset of $Q$. Thus we can define the minimal $A^e$-projective resolution of $A$, which is also the two-sided Anick resolution of $A$.
	
	\begin{proposition}$(\cite[~Proposition~2.2]{ALS})$
		Given a toupie algebra $A=\qa$, $E=kQ_0$ the minimal $A^e$-projective resolution $\mathbb{P}_*$ of $A$ is defined by
		$$\cdots\longrightarrow \mathbb{P}_n\stackrel{d_n}{\longrightarrow}\cdots\longrightarrow \mathbb{P}_2\stackrel{d_2}{\longrightarrow}\mathbb{P}_1 \stackrel{d_1}{\longrightarrow}\mathbb{P}_0\stackrel{\mu}{\longrightarrow}A\longrightarrow 0,$$
		where $\mathbb{P}_n=A\otimes_E \mathbbm{k}W^{(n-1)}\otimes_E A$, and the differential is given by
		$$\mu(a\otimes b)=ab,\quad a,b\in A,$$
		$$d_1([v])=v\otimes1-1\otimes v,\quad v\in Q_1,$$
		$$d_2([w_0|w_1])=\sum_{p_1p_2p_3\in \supp(\tip^{-1}(w_0w_1))}c(p_1p_2p_3)p_1\otimes p_2\otimes p_3,\quad w_0,p_2\in Q_1,\; w_0w_1\in\tip\mathcal{G},$$
		$$	d_n([w_0|\cdots|w_{n-1}])=\left\{
		\begin{array}{*{2}{ll}}
			\sum w^{(1)}\otimes w^{(2)}\otimes w^{(3)},& \quad \text{if $n$ is even,}\\
			w_{n-1}'[w_{n-2}'|\cdots|w_0']-[w_0|\cdots|w_{n-2}]w_{n-1},&\quad \text{if $n$ is odd.}
		\end{array}
		\right.$$
		where $ w^{(2)}\in W^{(n-2)}$, $w_{n-1}'\cdots w_0'=w_0\cdots w_{n-1}\in Q$ and $	[w_{n-1}'|w_{n-2}'|\cdots|w_0']$ is a right $(n-1)$-chain, while $[w_0|\cdots|w_{n-2}|w_{n-1}]$ is a left $(n-1)$-chain.
	\end{proposition}

	Now we applying the functor $E\otimes_A-$ to the resolution above, we get the minimal right projective resolution of $E$. Moreover, if we applying the functor $E\otimes_A-\otimes_A E$ to the resolution above, since the differential are all become zero, we will get the complex $\mathbb{T}_*:=\tor^A_*(E,E)$ with $\mathbb{T}_n=E\otimes_A \mathbb{P}_n\otimes_A E\cong \mathbbm{k}W^{(n-1)}$.
	
	Now we begin to consider the SDR datum discussed in Theorem \ref{morse}. Since the minimal $A^e$-projective resolution of toupie algebra $A$ is the complex reduced by the Morse matching $\mathcal{M}$ from the reduced bar resolution $\mathbb{B}_*$ of $A$, we have
	$$
	\begin{tikzcd}
		\mathbb{B}'_* \arrow[r, "p", shift left=2] \arrow["h"', loop, distance=2em, in=215, out=145] & \mathbb{T}_* \arrow[l, "i", shift left]
	\end{tikzcd}$$		
	with $\mathbb{B}'_*=E\otimes_A\mathbb{B}_*\otimes_A E$. Now we give a specific description about the chain maps above.
	
	We call a bar term $[w_0|\cdots|w_r]$ attached if for $i=\{1,\cdots,r\}$ , we have $w_{i-1}w_i\in I$. That means if $r>1$, we have $w_{i-1}w_i=0$ in $A$.
	Suppose $\gamma=[w_0|\cdots|w_r]$ is attached but not a chain. Then there is a largest $i_1$ such that $w_{i_1}=w_{i_1}'w_{i_1}''$ and such that $\eta^1=[w_0|\cdots|w_{i_1}']$ is a chain. It may happen that $i_1=0$, in which case $w_0'$ is simply the first starting arrow in $w_0$. We define:
	$$\gamma^1=(-1)^{i_1}[\eta^1|w_{i_1}''|w_{i_1+1}|\cdots|w_r],\quad \Gamma^1=[\eta^1|w_{i_1}''w_{i_1+1}|\cdots|w_r].$$
	If $\Gamma^1$ is a chain or zero, stop. Else, there is some largest $i_2>i_1$ such that, keeping in with the notation above, $\eta^2=[w_0|\cdots|w_{i_1}'|\cdots|w_{i_2}']$ is a chain. In this case, set
	$$\gamma^2=(-1)^{i_2}[\eta^2|w_{i_2}''|w_{i_2+1}|\cdots|w_r],\quad \Gamma^2=[\eta^2|w_{i_1}''w_{i_2+1}|\cdots|w_r].$$	
	Continuing in this way, we obtain terms $\gamma=\Gamma^0,\cdots,\Gamma^n$ and $\gamma^1,\cdots,\gamma^n$, where $\Gamma^n$ is either zero or chain.
	
	\begin{lemma}\label{sdr of tor}
		With the notation above, for an attached bar term $\gamma=[w_0|\cdots|w_r]$, we have 
		$$h(\gamma)=\sum_{i=1}^{n}\gamma^i,\quad p(\gamma)=\Gamma^n,$$
		and for a chain $w=[w_0|\cdots|w_r]$, we have
		$$i(w)=\left\{
		\begin{array}{*{2}{ll}}
		 w,& \quad r\neq 1,\\
			\sum_{p_1p_2\in \supp(\tip^{-1}(w_0w_1))}c(p_1p_2)[p_1|p_2],&\quad r=1,\;p_1\in Q_1.
		\end{array}
		\right.$$
	\end{lemma}
	
	\begin{proof}
		If $\gamma=[w_0|\cdots|w_r]$ is attached with $w_{i-1}w_i$ divided by monomial relations for all $1\leq i\leq r$, the result of $h(\gamma)$ and $p(\gamma)$ is same as the proof in \cite[~Lemma~3.2]{TAMA}. If there exists a non monomial relation in $w_0\cdots w_r$, then $r=1$ and $w_0w_1$ is a path from the vertex $0$ to $\omega$. If $w_0\in Q_1$, then $\gamma$ is a $1$-chain, thus $h(\gamma)=0$ and $p(\gamma)=\gamma$. If $w_0\notin Q_1$, let $w_0=w_0'w_0''$ with $w_0'\in Q_1$, same as the algorithm discussed above the lemma, there is an unique zigzag path from $[w_0|w_1]$ to $[w_0'|w_0''w_1]$. Thus $h(\gamma)=\sum_{i=1}^{n}\gamma^i,  p(\gamma)=\Gamma^n$ for all attached bar term $\gamma$.
		
		Now let $w=[w_0|\cdots|w_r]$ is a chain, if $w$ only contains monomial relations, there is no zigzag paths from $w$ to other bar term which has the form $[w_0'|\cdots|w_r']\in \overline{A}^{\otimes (r+1)}$.
		If $w$ contains non monomial relations, we have $w=[w_0|w_1]$ with $w_0\in Q_1$. For all paths $p_1p_2\in supp(\tip^{-1}(w_0w_1))$ with $p_1\in Q_1$, we have
\[\begin{tikzcd}
	\lbrack w_0|w_1\rbrack && \lbrack p_0|p_1\rbrack \\
	& \lbrack p_1p_2\rbrack
	\arrow["c(p_1p_2)", from=1-1, to=2-2]
	\arrow["1", dashed, from=2-2, to=1-3]
\end{tikzcd}\]
		and by Theorem \ref{morse}, we get the description about $i(w)$.
	\end{proof}
	
	\subsection{$A_\infty$-structure on $\mathbb{T}_*$}
	
	With the notation in above section, we have that $(\mathbb{B}_*',d,\Delta')$ is a DG associative coalgebra with
	$$\Delta'([a_1|\cdots|a_n])=\sum_{i=1}^{n-1}[a_1|\cdots|a_i]\otimes[a_{i+1}|\cdots|a_n].$$
	If $w=[w_0|\cdots|w_n]$ is a $n$-chain, we have $\Delta'(w)=\sum w^{(1)}\otimes w^{(2)}$ where $w^{(1)}$ is also a chain and $h(w)=0$ by Lemma \ref{sdr of tor}. By using Theorem \ref{trans coalg}, we get $\{\Delta_i\}$ on $\mathbb{T}_*$ makes it to be an $A_\infty$ coalgebra. 
	
	\begin{theorem}\label{co-A}
		Let $A=\qa$ be a toupie algebra, $n\geq 3$ and $\gamma$ is a $r$-chain in $\mathbb{T}_*$. When $\gamma$ does not contain non monomial relations,
		$$\Delta_n(\gamma)=\sum(-1)^N\gamma_1\otimes \cdots\otimes\gamma_n,$$
		with $N=r_1+\sum_{i=1}^{n-1}(n-i)r_i$ while $(\gamma_1,\cdots,\gamma_n)$ is a decomposition of $\gamma$ and $\gamma_i$ is a $r_i$-chain for all $1\leq i\leq n$. When $\gamma=[w_0|w_1]$ with $w_0w_1\in \tip\mathcal{G}_{nomon}$,
		$$\Delta_n(\gamma)=\sum_{a_1\cdots a_n\in \supp(\tip^{-1}(w_0w_1))}c(a_1\cdots a_n)a_1\otimes\cdots\otimes a_n$$
		with $a_1,\cdots,a_n\in Q_1$.
	\end{theorem}
	
	\begin{proof}
		When $\gamma$ is a $r$-chain in $\mathbb{T}_*$ and does not contain non monomial relations, same as the proof in \cite[~Theorem~3.6]{TAMA} and \cite[~Theorem~4.1]{TAMA}, we have the operation $\Delta_n$ is just a right comb and non-trivial when $(\gamma_1,\cdots,\gamma_n)$ is a decomposition of $\gamma$ where all $\gamma_i$ are chains. We just give a description about $N$ since we give a different differential in $\mathbb{B}_*'$. Actually, the sign of a right comb is $(-1)^{(n+1)n/2-1}$. Since $|\gamma_i|=r_i+1$, by Koszul sign rule, we have the following equation (mod $2$):
		$$
		\begin{array}{*{3}{lll}}
			N & =& (n+1)n/2-1+(2-(n-1)+1)(r_1+1)+r_1+\cdots+(2-2+1)(r_{n-2}+1)+r_{n-2}+r_{n-1}\\
			&=&(n+1)n/2-1+\sum_{i=1}^{n-1}(n-i)r_i+r_1+1+\sum_{i=1}^{n-1}(n-i)\\
			&=&r_1+\sum_{i=1}^{n-1}(n-i)r_i.
		\end{array}
		$$
		
		Now we consider  $\gamma=[w_0|w_1]$ with $w_0w_1\in \tip\mathcal{G}_{nomon}$, by Lemma \ref{sdr of tor}, we have
		$$i(\gamma)=\sum_{a_1\cdots a_m\in supp(\tip^{-1}(w_0w_1))}c(a_1\cdots a_m)[a_1|a_2\cdots a_m]$$
		with $a_1,\cdots,a_m\in Q_1$, and
		$$\Delta'i(\gamma)=\sum_{a_1\cdots a_m\in supp(\tip^{-1}(w_0w_1))}c(a_1\cdots a_m)[a_1]\otimes [a_2\cdots a_m].$$
		Actually, there are no element in $\mathcal{G}$ can divide $a_2\cdots a_m$, 
		$h([a_2\cdots a_m])=[a_2|a_3\cdots a_m].$
		By induction, the operation $\Delta_n$ is also the right comb. Since $[a_i]$ are all $0$-chains, the sign $(-1)^N=1$. Therefore, we get the result.
	\end{proof}
	
	Now it is time to consider the Yoneda algebra $\mathbb{E}^*=\ext_A^*(E,E)$ of $A$. Since $Q$ is a finite quiver, $\mathbb{T}_*$ is finite dimensional as a vector space. Thus we have an isomorphism between $\mathbb{E}^*$ and $(\mathbb{T}_*)^{\vee}$, which we called $D^*$ (pay attention to the Koszul signs):
	$$D^n(f_1\otimes\cdots\otimes f_n)(c_1\otimes\cdots\otimes c_n)=(-1)^{N'}f_1(c_1)\cdots f_n(c_n)$$
	where $N'=\sum_{i=2}^{n}(|c_1|+\cdots+|c_{i-1}|)|f_i|$, $f_1\otimes\cdots\otimes f_n\in (\mathbb{E}^*)^{\otimes n}$ and $c_1\otimes\cdots\otimes c_n\in (\mathbb{T}_*)^{\otimes n}$.
	
	Define $m_n:\;(\mathbb{E}^*)^{\otimes n}\rightarrow \mathbb{E}^*$ by
	$$m_n(f_1\otimes\cdots\otimes f_n)=(-1)^{n(|f_1|+\cdots+|f_n|)}D^n(f_1\otimes\cdots\otimes f_n)\Delta_n.$$
	If we call an $A_\infty$-algebra structure on $\mathbb{E}^*$ is canonical if it is quasi-isomorphism to the DG associative algebra $(\mathbb{B}'^*,d,\Delta^\vee\circ D^2)$. We have the following result directly.
	
	\begin{corollary}\label{Toupie-Ainf}
		There is a canonical $A_\infty$-algebra structure on $\mathbb{E}^*$ given as follows. 
		\begin{itemize}
			\item If $\gamma_i$ is $r_i$-chain in $\mathbb{T}^{r_i+1}$, $1\leq i\leq n$, $\gamma=\gamma_1\cdots\gamma_n$ is a $r=r_1+\cdots+r_n+1$-chain which does not contain non monomial relations, then 
			$$m_n(\gamma_1^\vee\otimes\cdots\otimes\gamma_n^\vee)=(-1)^M\gamma^\vee,$$
			with $M=r_1+\sum_{i=1}^{n}(n+i+1)r_i+\sum_{i<j}r_ir_j+n(n+1)/2$;
			
			\item If $a_1,\cdots,a_n$ are arrows, such that $a_1\cdots a_n$ is a branches involved in a non monomial relation, then
			$$m_n(a_1^\vee\otimes\cdots\otimes a_n^\vee)=\sum (-1)^{n(n+1)/2} c(a_1\cdots a_n)[w_0|w_1]^\vee$$
			when there exist a $\rho_i\in\mathcal{G}_{nomon}$, with $\tip(\rho_i)=w_0w_1$, $w_0\in Q_1$ and $c(a_1\cdots a_n)a_1\cdots a_n$ is a summand of $\rho_i$;
			
			\item Otherwise, the higher product is zero.
		\end{itemize}
	\end{corollary}
	
	\begin{proof}
		We just give an verification about $M$. By Theorem \ref{co-A} and the discussion above, we have the following equation (mod $2$):
		$$
		\begin{array}{*{3}{lll}}
			M & =& N+N'+n(|\gamma_1|+\cdots+|\gamma_n|)\\
			&=&r_1+\sum_{i=1}^{n-1}(n-i)r_i+\sum_{i=2}^{n}(|\gamma_1|+\cdots+|\gamma_{i-1}|)|\gamma_i|+n(|\gamma_1|+\cdots+|\gamma_n|)\\
			&=&r_1+\sum_{i=1}^{n}(n+i+1)r_i+\sum_{i<j}r_ir_j+n(n+1)/2.
		\end{array}
		$$
		Thus we get the result.
	\end{proof}
	Actually, since the differential on $\mathbb{E}^*$ is trivial, the product $m_2$ is associative by the definition of $A_\infty$-algebra. Thus $A^!:=(\mathbb{E}^*,m_2)$ is just the Yoneda algebra $\ext_A(A/rad(A),A/rad(A))$ of the toupie algebra $A$, which is also called the homological dual of $A$.

	\begin{example}\label{ex}
		Consider the toupie algebra $A=\qa$ with the quiver $Q$ is given by:
		$$
		
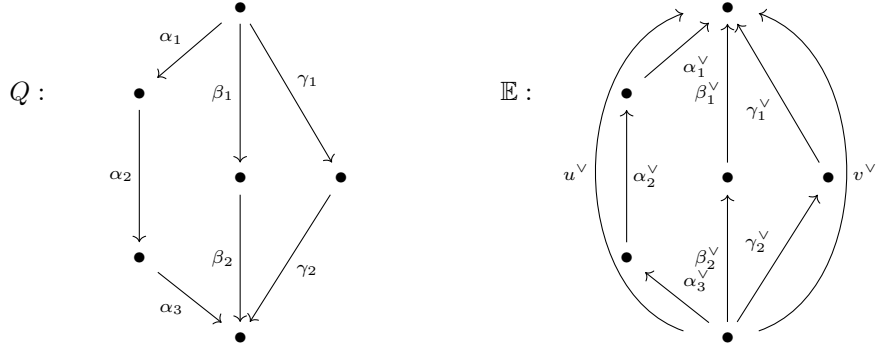
\begin{tikzcd}
			&                                 & \bullet \arrow[ld, "\alpha_1"'] \arrow[dd, "\beta_1"'] \arrow[rdd, "\gamma_1"] &                                 &  &             &                                      & \bullet                                                                                                                                                                                               &                                      \\
			Q: & \bullet \arrow[dd, "\alpha_2"'] &                                                                                &                                 &  & \mathbb{E}: & \bullet \arrow[ru, "\alpha_1^\vee"'] &                                                                                                                                                                                                       &                                      \\
			&                                 & \bullet \arrow[dd, "\beta_2"']                                                 & \bullet \arrow[ldd, "\gamma_2"] &  &             &                                      & \bullet \arrow[uu, "\beta_1^\vee"]                                                                                                                                                                    & \bullet \arrow[luu, "\gamma_1^\vee"] \\
			& \bullet \arrow[rd, "\alpha_3"'] &                                                                                &                                 &  &             & \bullet \arrow[uu, "\alpha_2^\vee"'] &                                                                                                                                                                                                       &                                      \\
			&                                 & \bullet                                                                        &                                 &  &             &                                      & \bullet \arrow[lu, "\alpha_3^\vee"'] \arrow[uu, "\beta_2^\vee"] \arrow[ruu, "\gamma_2^\vee"] \arrow[uuuu, "v^\vee"', bend right=71, shift right=2] \arrow[uuuu, "u^\vee", bend left=71, shift left=4] &                                     
		\end{tikzcd}$$
	and the ideal $I=\langle \alpha_{1}\alpha_{2}\alpha_{3}-\gamma_1\gamma_2,\;\beta_1\beta_2-\gamma_1\gamma_2\rangle$ with respect to the order $\alpha>\beta>\gamma$. Thus $\tip\mathcal{G}$ of $I$ is given by $\{\alpha_{1}\alpha_{2}\alpha_{3},\;\beta_1\beta_2\}$. Denote $u^\vee=([\alpha_{1}|\alpha_{2}\alpha_{3}])^\vee$ and $v^\vee=([\beta_1|\beta_2])^\vee$. By Corollary \ref{Toupie-Ainf}, we have some nonzero products on $\mathbb{E}^*$  which is given by:
	$$m_2(\beta_1^\vee,\beta_2^\vee)=-v^\vee,$$
	$$m_2(\gamma_1^\vee,\gamma_2^\vee)=u^\vee+v^\vee,$$
	$$m_3(\alpha_1^\vee,\alpha_2^\vee,\alpha_{3}^\vee)=u^\vee.$$

	\end{example}
		
	\subsection{More properties on Yoneda algebras}
		
	In \cite{GS}, they give a method to determine of the quiver algebra $A=\mathbbm{k}Q/I$ is a quasi-hereditary algebra. That reminds us to discuss some special case on toupie algebras. 
	
	\begin{proposition}
		Toupie algebras are quasi-hereditary.
	\end{proposition}
		
	\begin{proof}
		Let $A=\qa$ be a toupie algebra. Consider the associated monomial algebra $A_{Mon}=\mathbbm{k}Q/\langle \tip\mathcal{G}\rangle$. By \cite[~Theorem~3.10]{GS}, we can take $v_1=0$, $v_2=w$. Then $Q_{\widehat{v_1+v_2}}$ (the subquiver of Q obtained by removing the vertices $v_1,v_2$) is a disjoint union of quivers which are of type $\mathbb{A}$ with linear orientation. Thus they are quasi-hereditary. This leads to $A_{Mon}$ being quasi-hereditary. Moreover, by \cite[~Theorem~4.1]{GS}, $A$ is quasi-hereditary.
	\end{proof}	
	
	Actually, there are some interesting properties about quasi-hereditary. For example, \cite[~Theorem~2.2.1]{CPS} if an algebra $A$ is quasi-hereditary under some special conditions, the double homological dual of $A$ is isomorphic to its associated graded algebra $gr(A)$. Moreover, in this case, $gr(A)$ must be a Koszul algebra. To be more specific, for a quiver algebra $\qa$, $\qa$ is Koszul if and only if $I$ is generated by some quadratic relations (see for example \cite{BGS}). Now we have known that a toupie algebra $A$ is quasi-hereditary, we need to make a concrete characterization of $gr(A)$ for the following discussion.
		
	Recall the basic definition of $gr(A)$. Let $A$ be a finite dimensional algebra. Denote by $\mathfrak{r}$ the (Jacobson) radical $rad(A)$ of $A$. Then the graded algebra $gr(A)$ of $A$ associated with the radical filtration is defined as follows. As a graded vector space,
	$$gr(A)=A/\mathfrak{r}\oplus\mathfrak{r}/\mathfrak{r}^2\oplus\cdots\oplus\mathfrak{r}^t/\mathfrak{r}^{t+1}\oplus\cdots.$$
	The multiplication of $gr(A)$ is given as follows. For any two homogeneous elements:
	$x+\mathfrak{r}^{m+1}\in\mathfrak{r}^m/\mathfrak{r}^{m+1},\quad y+\mathfrak{r}^{n+1}\in\mathfrak{r}^n/\mathfrak{r}^{n+1},$
	we have $$(x+\mathfrak{r}^{m+1})\cdot(y+\mathfrak{r}^{n+1})=xy+\mathfrak{r}^{m+n+1}.$$

	In above discussions, for a fixed toupie algebra $A=\qa$, we have fixed a finite set of $r$ equations generating the non monomial relations in $I$ and having fixed the order of branches defined in $\{\alpha^{(a+l+m+1)},\cdots,\alpha^{(a+l+m+n)}\}$ with $l(\alpha^{(a+l+m+i)})\geq l(\alpha^{(a+l+m+j)})$ for all $i\leq j$. Let $C=(a_{ij})\in\mathbbm{k}^{r\times n}$ be the matrix whose rows are coefficients of each of these equations and this matrix is already reduced and its rank is $r$. Every non monomial relation will be of the form:
	$$\rho_i=\alpha^{(a+l+m+k_i)}+\sum_{j>k_i}a_{ij}\alpha^{(a+l+m+j)},\quad 1\leq i\leq r.$$
	That means the leftmost nonzero element of each row of the matrix $C$ is $1$ and all the elements above it are zeros. Now we give a special basis of $I$ in $\mathbbm{k}Q$ to help us construct $gr(A)$. We operate on the matrix $C$ by row transformations.
	\begin{itemize}
		\item Consider the rightmost column $k_0$ of $C_0:=C$ which is not zero and $a_{l_0k_0}$ is the bottom-most nonzero element in this column. Then reduced all $a_{l'k_0}$ with $l'<l_0$ to zero by $a_{l_0k_0}$ to get a new matrix $C_1$;
		
		\item Consider the rightmost column $k_1$ with $k_1<k_0$ of $C_1$ which is not zero. Let $a_{l_1k_1}$ be the bottom-most nonzero element in this column such that  all $a_{l_1k''}$ are zeros with $k''>k_1$. If no $a_{l_1l_1}$ fits this condition, the operation ends. If not, reduced all $a_{l'k_1}$ with $l'<l_1$ to zero by $a_{l_1k_1}$ to get a new matrix $C_2$;
		
		\item Repeat these operations. 
	\end{itemize}
	Since the matrix $C$ has $n$ columns with $n$ finite, this process will terminate and we will obtain a unique corresponding matrix $C'$. For example:
	$$C=\begin{pmatrix}
	1&&&&1&&1\\
	&1&&&&1&1\\
	&&1&&1&&\\
	&&&1&&1&	
	\end{pmatrix}\;\Longrightarrow
	\;C'=\begin{pmatrix}
		1&-1&-1&1&&&\\
		&1&&-1&&&1\\
		&&1&&1&&\\
		&&&1&&1&	
	\end{pmatrix}$$
	We note that $C'$ is also corresponding a class of non monomial generating set $\mathcal{G}_{nomon}'$ of $I$. We call the generating set $S=\mathcal{G}_{mon}\cup\mathcal{G}_{nomon}'$ the special basis of $I$ in $\mathbbm{k}Q$.
	
	\begin{lemma}\label{grA}
		Let $A=\qa$ be a toupie algebra with the special basis $S$ of $I$. For each relation $\rho\in S$, $$\rho=\beta_0+\sum_{i=1}^{k}c_i\beta_i$$ with $l(\beta_{i_1})\geq l(\beta_{i_2})$, $i_1\leq i_2$ and $c_0=1$. Then there exists $0\leq k_0\leq k$, such that for all $j\geq k_0$, $l(\beta_{j})=l(\beta_{k_0})$. Denote by $$\rho'=\beta_{k_0}+\sum_{i=k_0+1}^{k}c_{k_0}^{-1}c_i\beta_i.$$
		Then, the associated graded algebra $gr(A)$ is isomorphic to $\mathbbm{k}Q/I'$, where $I'$ is an admissible ideal whose generating relations are obtained from that of $I$ by replacing each $\rho$ by $\rho'$.
	\end{lemma}
		
	\begin{proof}
		It is obvious to see that each $\rho'$ is homogeneous and becomes zero in $gr(A)$ by its definition (for example, if $\rho'\in \mathfrak{r}^m\backslash\mathfrak{r}^{m+1}$, then $(\rho-c_{k_0}\rho')\in \mathfrak{r}^{m+1}$ becomes zero in $\mathfrak{r}^m/\mathfrak{r}^{m+1}$. Moreover, since $\rho=0$ in $A$, $\rho=0$ in $\mathfrak{r}^m$. That means $\rho'=c_{k_0}^{-1}(\rho-(\rho-c_{k_0}\rho'))=c_{k_0}^{-1}\rho=0$ in $\mathfrak{r}^m/\mathfrak{r}^{m+1}$). Moreover, $\dim_\mathbbm{k}(\mathbbm{k}Q/I')=\dim_\mathbbm{k}(A)=\dim_\mathbbm{k}(gr(A))$.
	\end{proof}
	
	Now we give the main result of this subsection. Recall the non oriented graph of each toupie algebra discussed in \cite[~Definition~4]{ALS}. For a toupie algebra $A=\qa$, the non oriented graph $\Gamma(A)=(\Gamma_0,\Gamma_1)$ is given by 
	$$\Gamma_0=\{\text{the branches in Q involved in non monomial relations};\}$$
	$$\Gamma_1=\{(\alpha^{(p)},\alpha^{(q)})\;|\;\text{there exists a relation that involves both $\alpha^{(p)}$ and $\alpha^{(q)}$}\}.$$
	
	\begin{theorem}\label{double-dual}
		Let $A=\qa$ be a toupie algebra, with $\mathcal{G}$ as the minimal generating set of $I$, $gr(A)$ as the associated graded algebra of $A$, and $\Gamma(A)$ as the non oriented graph of $A$. If all the monomial relations in $\mathcal{G}$ are quadratic, and all the non monomial relations involved in each connected component of $\Gamma(A)$ have a quadratic tip except for at most one exception, then $A^{!!}\cong gr(A)$.
	\end{theorem}
	
	\begin{proof}
		Since all the non monomial relations involved in each connected component of $\Gamma(A)$ have quadratic tip except at most one exception, by Corollary \ref{Toupie-Ainf}, the quiver of $A^!$ is equal to the opposite quiver of $Q$ and the multiplication $m_2$ in $A^!$ is always zero except to the quadratic paths involved in the minimal generating set $\mathcal{G}$ of $I$. Therefore, all the relations in $A^!$ are quadratic. By the formula in Corollary \ref{Toupie-Ainf}, only paths of length two in the relations affect the multiplication $m_2$ in $A^1$. Therefore, $A^!\cong (gr(A))^!$. Thus $A^{!!}\cong (gr(A))^{!!}$. Moreover, in this case, $gr(A)$ is Koszul, then $gr(A)\cong (gr(A))^{!!}$. 
	\end{proof}
	
	In this case, we can regard $gr(A)$ as the Ext complex of $A^!$ after performing the operations of transfer theory twice. Unfortunately, since we do not have a corresponding theory about the bar complex of an $A_\infty$-algebra, we lose the information about the higher products on $\mathbb{E}^*$ when we perform the second transformation. By Corollary \ref{Toupie-Ainf}, since $A^!$ is Koszul, the $A_\infty$-structure on the Ext complex of $A^!$ is trivial ($m_n=0$ for $n\neq 2$). Due to those lost information, we cannot reconstruct $A$ from $gr(A)$ solely based on  this simple characterization of $gr(A)$.
	
	\begin{example}
		We continue to compute the toupie algebra in Example \ref{ex}. In this example ,  all the non monomial relations involved in each connected component of $\Gamma(A)$ have quadratic tip except $\alpha_{1}\alpha_{2}\alpha_{3}-\gamma_1\gamma_2$. Thus $A^!$ is Koszul, and we have already computed that $A^!=\mathbbm{k}Q^{op}/\langle \alpha_{1}\alpha_{2},\alpha_{2}\alpha_{3}\rangle$. Now consider $A^{!!}$ whose quiver is $(Q^{op})^{op}=Q$. Actually, by Corollary \ref{Toupie-Ainf}, we have $m_2(\beta_1,\beta_2)=0$ and $m_2(\gamma_1,\gamma_2)=0$.
		Therefore, we have $A^{!!}=\mathbbm{k}Q/\langle\beta_1\beta_2,\gamma_1\gamma_2\rangle$
	
		By Lemma \ref{grA} , $gr(A)=\mathbbm{k}Q/\langle\beta_1\beta_2,\beta_1\beta_2-\gamma_1\gamma_2\rangle$.	
		Thus $A^{!!}=gr(A)$, just like Theorem \ref{double-dual} says.
	\end{example}

		%\newpage	

	\end{document}